\documentclass{amsart}

\usepackage{enumerate}
\usepackage{graphicx,graphics}
\usepackage{amsfonts}
\usepackage{amssymb}
\usepackage{amsthm}
\usepackage{amsmath}
\usepackage{mathrsfs}
\usepackage{hyperref}

\input{xy}
\xyoption{all}

\newtheorem{theorem}{Theorem}[section]
\newtheorem{lemma}[theorem]{Lemma}

\newtheorem{corollary}[theorem]{Corollary}

\newtheorem{proposition}[theorem]{Proposition}

\DeclareMathOperator{\Aff}{Aff}
\DeclareMathOperator{\Sim}{Sim}

\DeclareMathOperator{\diam}{diam}

\numberwithin{equation}{section}

\begin{document}

\title{Hyperspaces of convex bodies of constant width}
\author{Sergey Antonyan, Natalia Jonard-P\'erez and Sa\'ul Ju\'arez-Ord\'o\~nez}

% Information for the first author

\address{Departamento de  Matem\'aticas,
Facultad de Ciencias, Universidad Nacional Aut\'onoma de M\'exico, 04510 M\'exico Distrito Federal, M\'exico.}
\email{(S.\,Antonyan) antonyan@unam.mx}
\email{(N.\,Jonard-P\'erez) nat@ciencias.unam.mx}
\email{(S.\,Ju\'arez-Ord\'o\~nez) sjo@ciencias.unam.mx}

%General info

\thanks{The first, second and third authors were supported by CONACYT (Mexico) grants 000165195, 204028 and 220836, respectively.}

\begin{abstract}
For every $n\geq1$, let $cc(\mathbb{R}^n)$ denote the hyperspace of all non-empty compact convex subsets of the Euclidean space $\mathbb{R}^n$ endowed with the Hausdorff metric topology. For every non-empty convex subset $D$ of $[0,\infty)$ we denote by $cw_D(\mathbb{R}^n)$ the subspace of $cc(\mathbb{R}^n)$ consisting of all compact convex sets of constant width $d\in D$ and by $crw_D(\mathbb{R}^n)$ the subspace of the product $cc(\mathbb{R}^n)\times cc(\mathbb{R}^n)$ consisting of all pairs of compact convex sets of constant relative width $d\in D$. In this paper we prove that $cw_D(\mathbb{R}^n)$ and $crw_D(\mathbb{R}^n)$ are homeomorphic to $D\times\mathbb{R}^n\times Q$, whenever $D\neq\{0\}$ and $n\geq2$, where $Q$ denotes the Hilbert cube. In particular, the hyperspace $cw(\mathbb{R}^n)$ of all compact convex bodies of constant width as well as the hyperspace $crw(\mathbb{R}^n)$ of all pairs of compact convex sets of constant relative positive width are homeomorphic to $\mathbb{R}^{n+1}\times Q$. %These results hold for the hyperspaces of (pairs of) compact convex sets of constant (relative) Minkowski-width, for certain Minkowski spaces.
\end{abstract}

\maketitle

\maketitle \markboth{S.A. ANTONYAN,  N. JONARD-P\'EREZ AND   S.\,JU\'AREZ-ORD\'O\~NEZ}{HYPERSPACES OF  CONVEX BODIES OF CONSTANT WITH}

\section{Introduction}

Let $cc(\mathbb{R}^n)$, $n\geq1$, denote the hyperspace of all non-empty compact convex subsets of $\mathbb{R}^n$ endowed with the topology induced by the Hausdorff metric: 
$$\rho_H(Y,Z)=\max\left\{ \sup\limits_{z\in Z}\|z-Y\|,~~~~~\sup\limits_{y\in Y}\|y-Z\|\right\},\quad Y,Z\in cc(\mathbb{R}^n)$$ where $\|\cdot\|$ denotes the euclidean norm on $\mathbb{R}^n$, i.e., $$\|x\|^2=\sum_{i=1}^nx_i^2,\quad x=(x_1,x_2,...,x_n)\in\mathbb{R}^n.$$ %which is just the absolute value on $\mathbb{R}$ when $n=1$.

In case $n=1$, it is easy to see that %the map $(x,y)\mapsto[x-y,x+y]$ defines a homeomorphism between $\mathbb{R}\times[0,\infty)$ and
$cc(\mathbb{R})$ is homeomorphic to $\mathbb{R}\times[0,1)$ and for every $n\geq2$, it is well known that $cc(\mathbb{R}^n)$ is homeomorphic to the \textit{punctured Hilbert cube} $Q_0:=Q\backslash\{*\}$, where $Q:=[0,1]^\infty$ is the Hilbert cube (see \cite[Theorem~7.3]{Nadler}).

By a \textit{convex body} we mean a compact convex subset of $\mathbb{R}^n$ with non-empty interior. A compact convex set in $\mathbb{R}^n$ is said to be of constant width $d\geq0$, if the distance between any two of its parallel support hyperplanes is equal to $d$ (see e.g., \cite[Chapter~7, \S\,6]{Webster}).

%Let $S$ be a centrally symmetric (about the origin) convex body in $\mathbb{R}^n$. A compact convex set $Y$ in $\mathbb{R}^n$ is said to be of constant $S$-width $d\geq0$, if $$h_Y(u)+h_Y(-u)=dh_S(u)$$ for every $u\in\mathbb{S}^{n-1}$, where $h_Y$ and $h_S$ denote the support functions of $Y$ and $S$ respectively (see formula (\ref{sup}) below) and $$\mathbb{S}^{n-1}=\{x\in\mathbb{R}^n\mid\|x\|=1\}$$ is the unit sphere in $\mathbb{R}^n$ (see e.g., \cite{Eggleston1}).

As a generalization of compact convex sets of constant width, H. Maehara introduced in \cite{Maehara} the concept of pairs of compact convex sets of constant relative width and showed that these pairs share certain properties of compact convex sets of constant width. A pair $(Y,Z)\in cc(\mathbb{R}^n)\times cc(\mathbb{R}^n)$ is said to be of constant relative width $d\geq0$, if $$h_Y(u)+h_Z(-u)=d$$ for every $u\in\mathbb{S}^{n-1}$, where $h_Y$ and $h_Z$ denote the support functions of $Y$ and $Z$, respectively (see formula (\ref{sup}) below) and $$\mathbb{S}^{n-1}=\bigl\{x\in\mathbb{R}^n\,\big|\,\|x\|=1\bigr\}$$ is the unit sphere in $\mathbb{R}^n$.

%In \cite{Sallee}, G. T. Sallee extended Maehara's definition of pairs of compact convex sets of constant relative width to that of pairs of compact convex sets of constant relative width with respect to a centrally symmetric (about the origin) convex body $S$ in $\mathbb{R}^n$ (see formula \ref{a} below) and prove some analogous results to those of \cite{Maehara}.

%In \cite{Sallee}, G. T. Sallee extended Maehara's defined in \cite{Sallee} a pair $(Y,Z)\in cc(\mathbb{R}^n)\times cc(\mathbb{R}^n)$ to be a pair of constant relative $S$-width if $$h_Y(u)+h_Z(-u)=dh_S(u)$$ for every $u\in\mathbb{S}^{n-1}$.

For every non-empty convex subset $D$ of $[0,\infty)$ we denote by $cw_D(\mathbb{R}^n)$ the subspace of $cc(\mathbb{R}^n)$ consisting of all compact convex sets of constant width $d\in D$ and by $crw_D(\mathbb{R}^n)$ the subspace of the product $cc(\mathbb{R}^n)\times cc(\mathbb{R}^n)$ consisting of all pairs of compact convex sets of constant relative width $d\in D$.
We shall use $cw(\mathbb{R}^n)$ and $crw(\mathbb{R}^n)$   for $cw_{(0, \infty)}(\mathbb{R}^n)$ and  $crw_{(0, \infty)}(\mathbb{R}^n)$, respective�ly.

Note that if $D\subset(0,\infty)$, then every $A\in cw_D(\mathbb{R}^n)$ is a convex body. This does not hold in the hyperspace $crw_D(\mathbb{R}^n)$, i.e., there are pairs $(Y,Z)$ of compact convex sets of constant relative width $d>0$, such that either $Y$ or $Z$ is not a convex body. For instance, the pair $(\mathbb{B}^n,\{0\})$ is of constant width $1$, where $$\mathbb{B}^n=\bigl\{x\in\mathbb{R}^n\,\big|\,\|x\|\leq1\bigr\}$$ is the unit ball in $\mathbb{R}^n$, while $\{0\}$ is not a body.

Note also that if $D=\{0\}$, then the hyperspaces $cw_D(\mathbb{R}^n)=\bigl\{\{x\}\,\big|\,x\in\mathbb{R}^n\bigr\}$ and $crw_D(\mathbb{R}^n)=\bigl\{\bigl(\{x\},\{x\}\bigr)\,\big|\,x\in\mathbb{R}^n\bigr\}$ are both homeomorphic to $\mathbb{R}^n$.

In case $n=1$, it is easy to see that $cw_D(\mathbb{R})$ is homeomorphic to $D\times\mathbb{R}$ for every non-empty convex subset $D$ of $[0,\infty)$ and that $crw_D(\mathbb{R})$ is homeomorphic to $D\times\mathbb{R}\times[0,1]$ for every non-empty convex subset $D\neq\{0\}$ of $[0,\infty)$ (see Propositions \ref{cw1} and \ref{crw1}).

However, for every $n\geq2$, the topological structure of $cw_D(\mathbb{R}^n)$  and $crw_D(\mathbb{R}^n)$ had remained unknown, except for the cases of convex sets $D$ of the form $[d_0,\infty)$ with $d_0\geq0$, for which it was proved in \cite[Corollary~1.2]{Bazylevych2} (relying on \cite[Theorem~1.1]{Bazylevych2}) that $cw_D(\mathbb{R}^n)$ is  homeomorphic to the punctured Hilbert cube $Q_0$.

It is the purpose of this paper to  give a complete description of the topological structure of the hyperspaces $cw_D(\mathbb{R}^n)$ and $crw_D(\mathbb{R}^n)$ for every $n\geq2$ and every non-empty convex subset $D\neq\{0\}$ of $[0,\infty)$. Namely, we prove in Theorem~\ref{main1} that the hyperspace $cw_D(\mathbb{R}^n)$ is  homeomorphic to $D\times\mathbb{R}^n\times Q$. In particular, the hyperspace $cw(\mathbb{R}^n)$ of all convex bodies of constant width is homeomorphic to $\mathbb{R}^{n+1}\times Q$ (Corollary \ref{maincor}% and \ref{maincor2}
). Besides, we prove in Theorem~\ref{main2} that   the hyperspace  $crw_D(\mathbb{R}^n)$
 is homeomorphic to  $cw_D(\mathbb{R}^n)$ (in particular, $crw(\mathbb{R}^n)$ is homeomorphic to $cw(\mathbb{R}^n)$).

Our argument is based, among other things,  on \cite[Theorem~1.1]{Bazylevych2} which asserts that the hyperspace $cw_D(\mathbb{R}^n)$, with $D\ne\{0\}$, is a contractible Hilbert cube manifold. However, the proof of this result  given in \cite{Bazylevych2}  contains a gap. Namely, it is claimed within the proof that for any $n\geq3$ and any regular $n$-simplex $\Delta\subset\mathbb{R}^n$ of side length $d>0$, the intersection of all closed balls with centers at the vertices of $\Delta$ and radius $d$ is of constant width $d$. But, this claim  is not true. In fact, for every $n\geq3$, no finite intersection of balls in $\mathbb{R}^n$ is of constant width, unless it reduces to a single ball (see \cite[Corollary~3.3]{Lachand}). This shows a striking difference with the two-dimensional case, where the intersection of all closed discs in $\mathbb{R}^2$ of radius $d>0$ and centers at the vertices of an equilateral triangle of side length $d$, is a convex body of constant width $d$, which is well known under the name of {\it Reuleaux triangle} (see e.g., \cite[Chapter~7, \S\,6]{Webster}).  Fortunately, this gap can be filled in using \cite[Theorem~4.1]{Lachand}, which describes a method for constructing convex bodies of constant width in arbitrary dimension $n$, starting from a given projection in dimension $n-1$.  We  decided to present below in   Theorem~\ref{Correctionl}  a detailed correct proof of \cite[Theorem~1.1]{Bazylevych2}; this makes our presentation more complete.

\smallskip

\section{Preliminaries}

All maps between topological spaces are assumed to be continuous.

A metrizable space $X$ is called an \textit{absolute neighborhood retract} (denoted by $X\in\mathrm{ANR}$) if for any metrizable space $Z$ containing $X$ as a closed subset, there exist a neighborhood $U$ of $X$ in $Z$ and a retraction $r:U\to X$.

Recall that a map $f:X\to Y$ between topological spaces is called \textit{proper} if for every compact subset $K$ of $Y$, the inverse image $f^{-1}(K)$ is a compact subset of $X$. A proper map $f:X\to Y$ between ANR's is called \textit{cell-like} if it is onto and each point inverse $f^{-1}(y)$ has the property $UV^{\infty}$, i.e., for each neighborhood $U$ of $f^{-1}(y)$ there exists a neighborhood $V\subset U$ of $f^{-1}(y)$ such that the inclusion $V\hookrightarrow U$ is homotopic to a constant map of $V$ into $U$. In particular, if $f^{-1}(y)$ is  contractible, then it has the property $UV^{\infty}$ (see \cite[Chapter~XIII]{Chapman}).

We refer the reader to \cite{Eggleston}, \cite{Mozsynska}, \cite{Schneider} and \cite{Webster} for the theory of convex sets. However, we recall here some notions of convexity that will be used throughout the paper. We begin with the Minkowski operations.

For any subsets $Y$ and $Z$ of $\mathbb{R}^n$ and $t\in\mathbb{R}$, the sets $$Y+Z=\{y+z\mid y\in Y, z\in Z\}\quad\quad\textnormal{and}\quad\quad tY=\{ty\mid y\in\mathbb{R}^n\}$$ are called the \textit{Minkowski sum} of $Y$ and $Z$ and the \textit{product} of $Y$ by $t$, respectively. It is well known that these operations preserve compactness and convexity and are continuous with respect to the Hausdorff metric.

As usual, we denote by $C(\mathbb{S}^{n-1})$ the Banach space of all maps from $\mathbb{S}^{n-1}$ to $\mathbb{R}$ topologized by the supremum metric: $$\varrho(f,g)=\sup\bigl\{|f(u)-g(u)|\,\big|\,u\in\mathbb{S}^{n-1}\bigr\},\quad f,g\in C(\mathbb{S}^{n-1}).$$

The \textit{support function} of $Y\in cc(\mathbb{R}^n)$ is the map $h_Y\in C(\mathbb{S}^{n-1})$ defined by
\begin{equation}\label{sup}
h_Y(u)=\max\bigl\{\left\langle y,u\right\rangle\big|\,y\in Y\bigr\},\quad u\in\mathbb{S}^{n-1}
\end{equation}
where $\left\langle\,\,,\,\right\rangle$ denotes the standard inner product in $\mathbb{R}^n$. %Thus, for a unit vector $u\in\mathbb{S}^{n-1}$, $h_Y(u)$ measures the distance from the origin to the supporting hyperplane of $Y$ with exterior normal vector $u$.

For every $Y, Z\in cc(\mathbb{R}^n)$ and $\alpha, \beta\geq0$, the support function of $\alpha Y+\beta Z$ satisfies the following equality:
\begin{equation}\label{suppconvex}
h_{\alpha Y+\beta Z}=\alpha h_Y+\beta h_Z.
\end{equation}
(see e.g., \cite[Theorem~5.6.2]{Webster}).

It is well known that the map $\varphi:cc(\mathbb{R}^n)\to C(\mathbb{S}^{n-1})$ defined by
\begin{equation}\label{embedding1}
\varphi(Y)=h_Y,\quad Y\in cc(\mathbb{R}^n)
\end{equation}
is an affine isometric embedding and the image $\varphi\bigl(cc(\mathbb{R}^n)\bigr)$ is a locally compact closed convex subset of the Banach space $C(\mathbb{S}^{n-1})$ (see e.g., \cite[p.~57,~Note~6]{Schneider}). Here, the map $\varphi$ is affine with respect to the Minkowski operations, i.e., for any $Y,Z\in cc(\mathbb{R}^n)$ and $t\in[0,1]$, equality (\ref{suppconvex}) clearly implies that
\begin{equation*}
\varphi\bigl(tY+(1-t)Z\bigr)=t\varphi(Y)+(1-t)\varphi(Z).
\end{equation*}

The \textit{width function} of $Y\in cc(\mathbb{R}^n)$ is the map $w_Y\in C(\mathbb{S}^{n-1})$ defined by
\begin{equation}\label{widths}
w_Y(u)=h_Y(u)+h_Y(-u),\quad u\in\mathbb{S}^{n-1}.
\end{equation}

Thus, a compact convex set $Y$ in $\mathbb{R}^n$ has \textit{constant width} $d\geq0$ if $w_Y$ is the constant map with value $d$. Equivalently, if
\begin{equation}\label{equiv}
Y-Y=d\mathbb{B}^n=\bigl\{x\in\mathbb{R}^n\,\big|\,\|x\|\leq d\bigr\}
\end{equation}
where $-Y=\{-y\mid y\in Y\}$  (see e.g., \cite[Chapter 7, \S 6]{Webster}).

In general, we denote by $$B(x,r)=x+r\mathbb{B}^n=\bigl\{y\in\mathbb{R}^n\,\big|\,\|x-y\|\leq r\bigr\}$$ the closed ball with center $x\in\mathbb{R}^n$ and radius $r\geq0$.

The diameter of $Y\in cc(\mathbb{R}^n)$ is denoted by $\diam Y$. It is well known that the function $\diam:cc(\mathbb{R}^n)\to[0,\infty)$ defined by
\begin{equation*}
Y\mapsto\diam Y,\quad Y\in cc(\mathbb{R}^n)
\end{equation*}
is continuous (see e.g., \cite[Example~2.7.11]{Webster}). Clearly, if $Y\in cw_{[0,\infty)}(\mathbb{R}^n)$, then $\diam Y$ is just the width of $Y$. Let $\omega$ denote the restriction to $cw_{[0,\infty)}(\mathbb{R}^n)$ of the diameter map. Then
\begin{equation}\label{diam}
\omega:cw_{[0,\infty)}(\mathbb{R}^n)\to[0,\infty)
\end{equation}
is obviously continuous. Furthermore, for every $Y, Z\in cw_{[0,\infty)}(\mathbb{R}^n)$ and $t\in[0,1]$, one has
\begin{equation}\label{omegaafin}
\omega\bigl(tY+(1-t)Z\bigr)=t\omega(Y)+(1-t)\omega(Z).
\end{equation}
Indeed,
\begin{align*}
\bigl(tY+(1-t)Z\bigr)-\bigl(tY+(1-t)Z\bigr) &=t(Y-Y)+(1-t)(Z-Z) \\
&=t\,\omega(Y)\,\mathbb{B}^n+(1-t)\,\omega(Z)\,\mathbb{B}^n \\
&=\bigl(t\,\omega(Y)+(1-t)\,\omega(Z)\bigr)\mathbb{B}^n.
\end{align*}
Hence, according to equality (\ref{equiv}), $tY+(1-t)Z$ is of constant width $t\,\omega(Y)+(1-t)\,\omega(Z)$.

\smallskip

The concept of a compact convex set of constant width was extended by H. Maehara \cite{Maehara} to that of pairs of compact convex sets of constant relative width. Namely, a \textit{pair} $(Y,Z)$ of compact convex sets in $\mathbb{R}^n$ is said to be of \textit{constant relative width} $d\geq0$ if the map $w_{(Y,Z)}\in C(\mathbb{S}^{n-1})$ defined by
\begin{equation}\label{widthpairmap}
w_{(Y,Z)}(u)=h_Y(u)+h_Z(-u),\quad u\in\mathbb{S}^{n-1}
\end{equation}
is a constant map with value $d$. Equivalently, if
\begin{equation*}
Y-Z=d\mathbb{B}^n.
\end{equation*}

Obviously, a compact convex set $Y$ of $\mathbb{R}^n$ is of constant width $d\geq0$ if and only if $(Y,Y)$ is a pair of constant relative width $d\geq0$. From this fact, one gets a natural embedding $e:cw_D(\mathbb{R}^n)\to crw_D(\mathbb{R}^n)$ given by the rule:
\begin{equation}\label{embedding2}
e(Y)=(Y,Y),\quad Y\in cw_D(\mathbb{R}^n).
\end{equation}

We consider the Minkowski operations in $cc(\mathbb{R}^n)\times cc(\mathbb{R}^n)$, i.e., for every $t\in\mathbb{R}$ and $(Y,Z), (A,E)\in cc(\mathbb{R}^n)\times cc(\mathbb{R}^n)$, $$(Y,Z)+(A,E)=(Y+A,Z+E)\quad\textnormal{and}\quad t(Y,Z)=(tY,tZ).$$

It follows from equality (\ref{suppconvex}) and formula (\ref{embedding1}) that the map $$\varphi\times\varphi:cc(\mathbb{R}^n)\times cc(\mathbb{R}^n)\longrightarrow C(\mathbb{S}^{n-1})\times C(\mathbb{S}^{n-1})$$ defined by
\begin{equation}\label{embedding3}
(Y,Z)\mapsto (h_Y,h_Z),\quad Y,Z\in cc(\mathbb{R}^n)
\end{equation}
embeds $cc(\mathbb{R}^n)\times cc(\mathbb{R}^n)$ as a closed convex subset in the Banach space $C(\mathbb{S}^{n-1})\times C(\mathbb{S}^{n-1})$.

%Thus, we can define a map $\vartheta:crw_{[0,\infty)}(\mathbb{R}^n)\to[0,\infty)$ by the rule:
%\begin{equation}\label{Width}
%\vartheta(Y,Z)=\frac{1}{2}\omega(Y+Z),\quad (Y,Z)\in crw_{[0,\infty)}(\mathbb{R}^n).
%\end{equation}
%It then follows that whenever $(Y,Z)\in crw_D(\mathbb{R}^n)$ for some convex subset $D$ of $[0,\infty)$, the image $\vartheta(Y,Z)\in D$, i.e., $\vartheta(Y,Z)$ is just the width of the pair $(Y,Z)$. Also, equality (\ref{omegaafin}) easily implies that the map $\vartheta$ defined by formula (\ref{Width}) satisfies the following equality
%\begin{equation}\label{varafin}
%\vartheta\bigl(t\,(Y,Z)+(1-t)(A,E)\bigr)=t\,\vartheta(Y,Z)+(1-t)\,\vartheta(A,E).
%\end{equation}
%for all $(Y,Z), (A,E)\in crw_{[0,\infty)}(\mathbb{R}^n)$ and $t\in[0,1]$.
%For every $n\geq1$, we denote by $\mathcal{B}(n)$ the subspace of $cc(\mathbb{R}^n)$ consisting of all centrally symmetric (about the origin) convex bodies of $\mathbb{R}^n$, i.e., $\mathcal{B}(n)$ consists of all convex bodies $S\in cc(\mathbb{R}^n)$ such that $S=-S$.
%It is a well-known classical fact that for every $n\geq1$, every centrally symmetric (about the origin) convex body $S$ in $\mathbb{R}^n$ is the unit ball of a Minkowski space (finite-dimensional Banach space) (see e.g., \cite{Rudin} or \cite{Thompson}). %and that the full general linear group $GL(n)$ of all non-singular linear transformations of $\mathbb{R}^n$ acts continuously on $\mathcal{B}(n)$ (see \cite{Antonyan}).

\bigskip

%, i.e., $Scw_D(\mathbb{R}^n)$ consists of all $Y\in cc(\mathbb{R}^n)$ such that
%\begin{equation}\label{Mwidth}
%h_Y(u)+h_Y(-u)=h_{dS}(u),
%\end{equation}
%for every $u\in\mathbb{S}^{n-1}$ and some $d\in D$. Equivalently, if
%\begin{equation*}
%Y-Y=dS
%\end{equation*}
%for some $d\in D$ (see~\cite{Sallee}).
%, i.e., $Scrw_D(\mathbb{R}^n)$ consists of all pairs $(Y,Z)\in cc(\mathbb{R}^n)\times cc(\mathbb{R}^n)$ such that
%\begin{equation}\label{Mpwidth}
%h_Y(u)+h_Z(-u)=h_{dS}(u),
%\end{equation}
%for every $u\in\mathbb{S}^{n-1}$ and some $d\in D$. Equivalently, if
%\begin{equation*}
%Y-Z=dS
%\end{equation*}
%for some $d\in D$, (see~\cite{Sallee}).

Next we recall that the group $\Aff(n)$ of all \textit{affine transformations} of $\mathbb{R}^n$ is defined to be the (internal) semidirect product: $$\mathbb R^{n}\rtimes GL(n)$$ where $GL(n)$ is the group of all non-singular linear transformations of $\mathbb{R}^n$ endowed with the topology inherited from $\mathbb{R}^{n^2}$  (see e.g. \cite[p.~102]{Alperin}). As a semidirect product, $\Aff(n)$ is topologized by the product topology of $\mathbb R^{n}\times GL(n)$, becoming a Lie group with two connected components. Each element $g\in\Aff(n)$ is usually represented by $g=T_v+\sigma$, where $\sigma\in GL(n)$ and $T_v:\mathbb R^{n}\to\mathbb R^{n}$ is the translation by $v\in\mathbb{R}^n$, i.e., $$T_v(x)=v+x,\quad x\in\mathbb{R}^n$$ and thus,  $$g(x)=v+\sigma(x),\quad x\in\mathbb{R}^n.$$ Note that for every $g\in\Aff(n)$, $t\in\mathbb{R}$ and $x,y\in\mathbb{R}^n$, we have that
\begin{equation}\label{remark}
g\bigl(tx+(1-t)y\bigr)=tg(x)+(1-t)g(y).
\end{equation}
Indeed,
\begin{align*}
  g\bigl(tx+(1-t)y\bigr) &=v+\sigma\bigl(tx+(1-t)y\bigr) \\
   &=tv+(1-t)v+t\sigma(x)+(1-t)\sigma(y) \\
   &=t\bigl(v+\sigma(x)\bigr)+(1-t)\bigl(v+\sigma(y)\bigr) \\
   &=tg(x)+(1-t)g(y).
\end{align*}

A map $g:\mathbb{R}^n\to\mathbb{R}^n$ is called a \textit{similarity transformation} of $\mathbb{R}^n$, if there is a $\lambda>0$, called the \textit{ratio} of $g$, such that $$\|g(x)-g(y)\|=\lambda\|x-y\|$$ for every $x,y\in\mathbb{R}^n$ (see e.g. \cite[Chapter~4, \S\,1]{Mozsynska}). Clearly, every similarity transformation $g:\mathbb{R}^n\to\mathbb{R}^n$ with ratio $\lambda$ is an affine transformation of $\mathbb{R}^n$. Indeed, such a $g$ is just the composition of the homothety with center at the origin and ratio $\lambda$ and the isometry $\frac{1}{\lambda}g$ (see \cite[p.~xii]{Schneider}). Let $\Sim(n)$ denote the closed subgroup of $\Aff(n)$ consisting of all similarity transformations of $\mathbb{R}^n$. %, i.e., $\Sim(n)$ is the group generated by all the homotheties and isometries of $\mathbb{R}^n$. Each element $g\in\Sim(n)$ is the composition of a homothety and an isometry.
Clearly, the natural action of $\Sim(n)$ on $\mathbb{R}^n$ given by the evaluation map
\begin{equation}\label{action1}
(g,x)\mapsto gx:=g(x),\quad g\in\Sim(n),\quad x\in\mathbb{R}^n
\end{equation}
is continuous (see e.g., \cite[Proposition~2.6.11 and Theorem~3.4.3]{Engelking}). This action induces a continuous action on the hyperspace $cc(\mathbb{R}^n)$, which is given by the rule:
\begin{equation}\label{action2}
(g,Y)\mapsto gY=\{gy\mid y\in Y\},\quad g\in\Sim(n),\quad Y\in cc(\mathbb{R}^n)
\end{equation}
(see \cite[Proposition 3.1]{AntJon}).

It is a well known classical fact that for every $Y\in cc(\mathbb{R}^n)$, there is a unique ball $\mathscr{B}(Y)\subset\mathbb{R}^n$ of minimal radius $\mathscr{R}(Y)$ containing $Y$ (see, e.g., \cite[Theorem~12.7.5]{Mozsynska}). The ball $\mathscr{B}(Y)$  is known under the name of {\it Chebyshev ball of}\, $Y$. In this case, the center of $\mathscr{B}(Y)$ belongs to $Y$, and we will denote it by $\mathscr{C}(Y)$. By \cite[Corollary~12.7.6]{Mozsynska}, the function $\mathscr{C}:cc(\mathbb{R}^n)\to\mathbb{R}^n$ defined by
\begin{equation}\label{Chevyshev}
Y\mapsto\mathscr{C}(Y),\quad Y\in cc(\mathbb{R}^n)
\end{equation}
is continuous. Furthermore, if $cc(\mathbb{R}^n)$ and $\mathbb{R}^n$ are endowed with the actions (\ref{action2}) and (\ref{action1}), respectively, then clearly the map $\mathscr{C}$ is also $\Sim(n)$-equivariant, i.e., $$\mathscr{C}(gY)=g\mathscr{C}(Y)$$ for every $g\in\Sim(n)$ and $Y\in cc(\mathbb{R}^n)$ (see e.g., \cite[Exercise 12.20]{Mozsynska}).

\smallskip

A \textit{Hilbert cube manifold} or a $Q$-\textit{manifold} is a separable, metrizable space that admits an open cover each member of which is homeomorphic to an open subset of the Hilbert cube $Q$. Every $Q$-manifold is \textit{stable}, i.e., it is homeomorphic to its product with the Hilbert cube $Q$ (see \cite[Theorem~15.1]{Chapman}). We refer the reader to \cite{Chapman} and \cite{Van Mill} for the theory of $Q$-manifolds.

Finally, we recall Edwards' Theorem (see \cite[Theorem~43.1]{Chapman}), which by the stability Theorem for $Q$-manifolds, %(see \cite[Theorem~15.1]{Chapman}),
can be stated in the following form:

\begin{theorem}\label{Edwards}
\textnormal{(R.D. Edwards)} Let $X$ be a $Q$-manifold and $Y$ a locally compact \rm{ANR}. If there exists a cell-like map $f:X\to Y$ then $X$ is homeomorphic to the product $Y\times Q$.
\end{theorem}

\smallskip

\section{The hyperspaces $cw_D(\mathbb{R}^n)$}

In this section we describe for every $n\geq1$, the topology of the hyperspaces $cw_D(\mathbb{R}^n)$. We begin with the case $n=1$.

\begin{proposition}\label{cw1}
Let $D$ be a non-empty convex subset of $[0,\infty)$. Then the hyperspace $cw_D(\mathbb{R})$ is homeomorphic to $D\times\mathbb{R}$.
\end{proposition}

\begin{proof}
Define a map $f:cw_D(\mathbb{R})\to D\times\mathbb{R}$ by the rule: $$f\bigl([x,y]\bigr)=\bigl(y-x,(x+y)/2\bigr),\quad[x,y]\in cw_D(\mathbb{R}).$$ A simple calculation shows that $f$ is a homeomorphism.
\end{proof}

Now we assume, for the rest of the section, that $n\geq2$.

\begin{lemma}\label{centersum}
Let $Y\in cc(\mathbb{R}^n)$ be such that $\mathscr{C}(Y)=0$. Then $\mathscr{C}(Y+B)=0$ for every closed ball $B=B(0,r)$.
\end{lemma}

\begin{proof}
Let $\delta=\mathscr{R}(Y)$ and $\varepsilon=\mathscr{R}(Y+B)$ be the radii of $\mathscr{B}(Y)$ and $\mathscr{B}(Y+B)$, respectively. Then $\mathscr{B}(Y)+B$ is just the closed ball $B(0,\delta+r)$. Since $Y\subset\mathscr{B}(Y)$ and the Minkowski addition preserves inclusions, we get that $Y+B\subset\mathscr{B}(Y)+B$. By minimality of the radius $\varepsilon=\mathscr{R}(Y+B)$, we have that $\varepsilon\leq\delta+r$. Let $$O=\bigl\{x\in\mathscr{B}(Y+B)\,\big|\,B(x,r)\subset\mathscr{B}(Y+B)\bigr\}.$$ Clearly, $O=B(z,\varepsilon-r)$, where $z=\mathscr{C}(Y+B)$. Since $Y\subset Y+B\subset\mathscr{B}(Y+B)$ and since for every $y\in Y$, $$B(y,r)\subset\bigcup_{\gamma\in Y}B(\gamma,r)=\bigcup_{\gamma\in Y}\bigl(\gamma+B(0,r)\bigr)=Y+B\subset\mathscr{B}(Y+B)$$ we infer that $Y\subset O$. By minimality of $\mathscr{B}(Y)$, we have that $\delta\leq\varepsilon-r$. Consequently, $\varepsilon=\delta+r$. Uniqueness of $\mathscr{B}(Y+B)$ yields that $z=\mathscr{C}(Y+B)=0$, as required.
\end{proof}

\begin{lemma}\label{samecenter}
Let $B$ be a closed ball with center $y\in\mathbb{R}^n$ and let $Y\in cc(\mathbb{R}^n)$ be such that $\mathscr{C}(Y)=y$. Then $\mathscr{C}\bigl(tY+(1-t)B\bigr)=y$ for every $t\in\mathbb{R}$.
\end{lemma}

\begin{proof}
Let $r$ be the radius of $B$. Define $g:\mathbb{R}^n\to\mathbb{R}^n$ by $$gx=x-y,\quad x\in\mathbb{R}^n.$$ Then clearly $g\in\Sim(n)$ and
\begin{equation}\label{iff}
\mathscr{C}\bigl(tY+(1-t)B\bigr)=y\quad\Longleftrightarrow\quad g\mathscr{C}\bigl(tY+(1-t)B\bigr)=0.
\end{equation}
By the $\Sim(n)$-equivariance of $\mathscr{C}$, we have that $\mathscr{C}(tgY)=tg\mathscr{C}(Y)=tgy=t\cdot 0=0$. Note also that $$(1-t)gB=B\bigl(0,(1-t)r\bigr).$$ Then, again by the $\Sim(n)$-equivariance of $\mathscr{C}$, equality (\ref{remark}) and Lemma~\ref{centersum}, one gets $$g\mathscr{C}\bigl(tY+(1-t)B\bigr)=\mathscr{C}\bigl(tgY+(1-t)gB\bigr)=0.$$ Finally, by the equivalence (\ref{iff}), we get $\mathscr{C}\bigl(tY+(1-t)B\bigr)=y$.
\end{proof}

\begin{lemma}\label{closed}
For every non-empty closed subset $K$ of $[0,\infty)$, the hyperspace $cw_K(\mathbb{R}^n)$ of all compact convex sets of constant width $k\in K$ is closed in $cc(\mathbb{R}^n)$.
\end{lemma}

\begin{proof}
Let $(Y_i)_{i=1}^\infty$ be a sequence in $cw_K(\mathbb{R}^n)$ such that $Y_i\leadsto Y\in cc(\mathbb{R}^n)$. Then, by \cite[Theorem~1.8.11]{Schneider}, we have that $h_{Y_i}\leadsto h_Y$ and hence, for the width map (\ref{widths}) we get that $w_{Y_i}\leadsto w_Y$. Since for every $i\geq1$, $w_{Y_i}$ is a constant map with value in $K$ and $K$ is closed, $w_Y$ is also a constant map with value in $K$. Thus, $Y\in cw_K(\mathbb{R}^n)$, as required.
\end{proof}

\begin{lemma}\label{conv}
If $D$ is a non-empty convex subset of $[0,\infty)$, then the hyperspace $cw_D(\mathbb{R}^n)$ is convex with respect to the Minkowski operations.
\end{lemma}

\begin{proof}
This follows directly from the convexity of $D$ and equality (\ref{omegaafin}).
\end{proof}

The proof of the following theorem is essentially the same as the one presented in \cite[Theorem~1.1]{Bazylevych2}. Here we just fill the gap in that  proof.% of \cite[Theorem~1.1]{Bazylevych2}.

\begin{theorem}\label{Correctionl}
For every $d>0$, the hyperspace $cw_d(\mathbb{R}^n)$ of all convex bodies of constant width $d$, is a contractible $Q$-manifold.
\end{theorem}

\begin{proof}
By formula (\ref{embedding1}) and Lemmas \ref{closed} and \ref{conv}, the hyperspace $cw_d(\mathbb{R}^n)$ embeds as a locally compact closed convex subset in the Banach space $C(\mathbb{S}^{n-1})$. Hence, according to \cite[Theorem~7.1]{BP}, $cw_d(\mathbb{R}^n)$ is homeomorphic to either $\mathbb{R}^m\times[0,1]^p$ or $[0,1)\times[0,1]^p$ for some $0\leq m<\infty$ and $0\leq p\leq\infty$.

Now, for the case $n=2$, let $K$ denote the Reuleaux triangle in $\mathbb{R}^2$ that is the intersection of the closed discs of radius $d$, centered at the points $(0,0), (d,0)$ and $(d/2,d\sqrt{3}/2)$ of the plane $\mathbb{R}^2$. For any $\alpha\in[0,2\pi]$, denote by $K_\alpha$ the image of $K$ under a counterclockwise rotation by an angle $\alpha$ around the origin. Note that $$\bigl\{K_\alpha\,\big|\,\alpha\in[0,2\pi]\bigr\}\subset cw_d(\mathbb{R}^2).$$ Using formula (\ref{embedding1}), we identify the family $\bigl\{K_\alpha\,\big|\,\alpha\in[0,2\pi]\bigr\}$ with the family $$\bigl\{h_{K_\alpha}\,\big|\,\alpha\in[0,2\pi]\bigr\}$$ of the support functions of the sets $K_\alpha$, $\alpha\in[0,2\pi]$. Next, we show that the latter family contains linearly independent sets of arbitrary cardinality. Identify the circle $\mathbb{S}^1$ with the subset $\bigl\{e^{it}\,\big|\,t\in[0,2\pi]\bigr\}$ of the complex plane. Since orthogonal transformations preserve the inner product, it follows from the definition of the support functions that the following equality holds for every $\alpha$ and  $t$ in $[0,2\pi]$:
\begin{equation}\label{Kalpha}
h_{K_\alpha}(e^{it})=h_K\bigl(e^{i(t-\alpha)}\bigr).
\end{equation}
Elementary geometric arguments show that
\begin{equation}\label{cases}
h_K(e^{it}) =
\begin{cases}
d & \text{if }\quad t\in[0,\pi/3] \\
0 & \text{if }\quad t\in[\pi,4\pi/3] \\
x\in(0,d) & \text{if}\quad\,\,t\in(\pi/3,\pi)\cup(4\pi/3,2\pi).
\end{cases}
\end{equation}
Now, fixing $l\in\mathbb{N}$, we define for every $j=0,1,\dots,l-1$, the map $$h_j:=h_{K_{\frac{j\pi}{3l}}}\in\bigl\{h_{K_\alpha}\,\big|\,\alpha\in[0,2\pi]\bigr\}.$$ To see that the set $\{h_j\mid j=0,1,\dots,l-1\}$ is linearly independent, let $$g=\sum_{j=0}^{l-1}\lambda_jh_j$$ be a linear combination such that $g=0$. Using equalities (\ref{Kalpha}) and (\ref{cases}), we get that $$h_j(e^{i\frac{\pi}{3}})=h_K\bigl(e^{i(\frac{\pi}{3}-\frac{j\pi}{3l})}\bigr)=d$$ for every $j=0,1,\dots,l-1$. Hence, $$0=g(e^{i\frac{\pi}{3}})=\sum_{j=0}^{l-1}\lambda_jh_j(e^{i\frac{\pi}{3}})=d\sum_{j=0}^{l-1}\lambda_j$$ and consequently, $\sum_{j=0}^{l-1}\lambda_j=0$. Now, again by equalities (\ref{Kalpha}) and (\ref{cases}) we get for every $j=1,2,\dots,l-1$ that $$h_j\bigl(e^{i(\frac{\pi}{3}+\frac{\pi}{3l})}\bigr)=h_K\bigl(e^{i(\frac{\pi}{3}-\frac{(j-1)\pi}{3l})}\bigr)=d\qquad\textnormal{and}\qquad h_0\bigl(e^{i(\frac{\pi}{3}+\frac{\pi}{3l})}\bigr)\in(0,d)$$ Hence,
\begin{align*}
0=g\bigl(e^{i(\frac{\pi}{3}+\frac{\pi}{3l})}\bigr) & =\lambda_0h_0\bigl(e^{i(\frac{\pi}{3}+\frac{\pi}{3l})}\bigr)+\sum_{j=1}^{l-1}\lambda_jh_j\bigl(e^{i(\frac{\pi}{3}+\frac{\pi}{3l})}\bigr) \\
& =\lambda_0h_0\bigl(e^{i(\frac{\pi}{3}+\frac{\pi}{3l})}\bigr)+d\sum_{j=1}^{l-1}\lambda_j=\Bigl(d-h_0\bigl(e^{i(\frac{\pi}{3}+\frac{\pi}{3l})}\bigr)\Bigr)\sum_{j=1}^{l-1}\lambda_j
\end{align*}
and consequently, $\lambda_0=-\sum_{j=1}^{l-1}\lambda_j=0$.

Repeating the argument but evaluating the map $g$ at the points $e^{i(\frac{\pi}{3}+\frac{s\pi}{3l})}$ for $s=2,3,\dots,l-1$, we conclude that $\lambda_j=0$ for every $j=0,1,\dots,l-1$, and hence, the set $\{h_j\mid j=0,1,\dots,l-1\}$ is linearly independent. This yields that, $\bigl\{K_\alpha\,\big|\,\alpha\in[0,2\pi]\bigr\}$ is infinite-dimensional. %Thus, the hyperspace $cw_d(\mathbb{R}^2)$ is infinite-dimensional. %and therefore, homeomorphic either to $\mathbb{R}^p\times Q$ or to $[0,1)\times Q$ for some $0\geq p<\infty$. In any case, it is a contractible $Q$-manifold.

Now, for any $n\geq3$, denote by p$_2:\mathbb{R}^n\to\mathbb{R}^2$ the cartesian projection, i.e., $$\textnormal{p}_2\bigl((x_1,\dots,x_n)\bigr)=(x_1,x_2),\qquad (x_1,\dots,x_n)\in\mathbb{R}^n.$$ Denote also by $R(2)$ the family of all Reuleaux triangles in $\mathbb{R}^2$ of constant width $d$. Applying inductively the raising-dimension process described in \cite[Theorem~4.1]{Lachand} to every $Z\in R(2)$, we obtain the family $R(n)$ of all convex bodies $Y\subset\mathbb{R}^n$ of constant width $d$, such that p$_2(Y)\in R(2)$. Here we are considering $$\mathbb{R}^{n-1}=\bigl\{(x_1,x_2,\dots,x_n)\in\mathbb{R}^n\,\big|\,x_n=0\bigr\}$$ as the affine hyperplane of $\mathbb{R}^n$ in which the $(n-1)$-dimensional convex body of constant width $d$ is contained (see \cite[Theorem~4.1]{Lachand}). Then clearly,  $$\bigl\{K_\alpha\,\big|\,\alpha\in[0,2\pi]\bigr\}\subset R(2)=\bigl\{\textnormal{p}_2(Y)\,\big|\,Y\in R(n)\bigr\}=\textnormal{p}_2\bigl(R(n)\bigr).$$ Consequently, the space $R(n)$ is infinite-dimensional. Thus we conclude that the hyperspace $cw_d(\mathbb{R}^n)$ is also infinite-dimensional for every $n\geq2$. Then, by virtue of \cite[Theorem~7.1]{BP}, $cw_d(\mathbb{R}^n)$ is homeomorphic to either $\mathbb{R}^m\times Q$ or $[0,1)\times Q$ for some $0\leq m<\infty$. In either case, it is a contractible $Q$-manifold. This completes the proof.
\end{proof}

\begin{lemma}[{\rm\cite[Theorem 1.1]{Bazylevych2}}]\label{Correction} Let $D\neq\{0\}$ be a non-empty convex subset of $[0,\infty)$. Then the hyperspace $cw_D(\mathbb{R}^n)$ is a contractible $Q$-manifold.
\end{lemma}

\begin{proof}
It follows from Theorem \ref{Correctionl} that the hypespace $cw_D(\mathbb{R}^n)$ is infinite-dimensional. By Lemma \ref{conv}, it is convex and hence, contractible. It remains to show that $cw_D(\mathbb{R}^n)$ is a $Q$-manifold.

If $D$ is closed, then by Lemma \ref{closed}, $cw_D(\mathbb{R}^n)$ is closed in $cc(\mathbb{R}^n)$, and therefore, it is locally compact. Then the map $\varphi$ defined by formula (\ref{embedding1}) embeds $cw_D(\mathbb{R}^n)$ as a locally compact closed convex subset in the Banach space $C(\mathbb{S}^{n-1})$. Thus, by \cite[Theorem~7.1]{BP}, $cw_D(\mathbb{R}^n)$ is homeomorphic to either $\mathbb{R}^m\times Q$ or $[0,1)\times Q$ for some $0\leq m<\infty$. In either case, $crw_D(\mathbb{R}^n)$ is a $Q$-manifold.

Next, if $D$ is open, then $K_D:=[0,\infty)\backslash D$ is closed. By Lemma \ref{closed}, $cw_{K_D}(\mathbb{R}^n)$ is closed in $cc(\mathbb{R}^n)$, and hence, also in $cw_{[0,\infty)}(\mathbb{R}^n)$. Equivalently, $cw_D(\mathbb{R}^n)$ is open in $cw_{[0,\infty)}(\mathbb{R}^n)$, which by the above paragraph is a $Q$-manifold. Thus, we infer that $cw_D(\mathbb{R}^n)$ is also a $Q$-manifold.

Finally, let $D$ be a half-open interval properly contained in $[0,\infty)$. Assume without loss of generality that $D=[a,b)$ with $b>a\geq0$. Then $D$ is open in $[a,b]$ and consequently, $cw_D(\mathbb{R}^n)$ is open in $cw_{[a,b]}(\mathbb{R}^n)$. Since $cw_{[a,b]}(\mathbb{R}^n)$ is a $Q$-manifold, it then follows that $cw_D(\mathbb{R}^n)$ is also a $Q$-manifold. This completes the proof.
\end{proof}

Now, with the aid of the maps (\ref{diam}) and (\ref{Chevyshev}), for every convex subset $D\neq\{0\}$ of $[0,\infty)$, we can define a map $\eta_D:cw_D(\mathbb{R}^n)\to D\times\mathbb{R}^n$ by the rule:
\begin{equation}\label{eta}
\eta_D(Y)=\bigl(\omega(Y),\mathscr{C}(Y)\bigr),\quad Y\in cw_D(\mathbb{R}^n).
\end{equation}

%If $D=[0,\infty)$, then $\eta_D$ will simply be denoted by $\eta$.

\begin{proposition}\label{propCE}
The function $\eta_D:cw_D(\mathbb{R}^n)\to D\times\mathbb{R}^n$ defined by formula \textnormal{(\ref{eta})} is a cell-like map.
\end{proposition}

\begin{proof}
Continuity of $\eta_D$ follows from the continuity of $\omega$ and $\mathscr{C}$. Let $(d,x)\in D\times\mathbb{R}^n$ and $B=B(x,d/2)\in cw_D(\mathbb{R}^n)$. Then $\eta_D(B)=(d,x)$, and hence, $\eta_D$ is a surjective map.  We claim that the inverse image $\eta_D^{-1}\bigl((d,x)\bigr)$ is contractible. Indeed, define a homotopy $H:\eta_D^{-1}\bigl((d,x)\bigr)\times[0,1]\to\eta_D^{-1}\bigl((d,x)\bigr)$ by the Minkowski sum: $$H(A,t)=tA+(1-t)B,\quad A\in\eta_D^{-1}\bigl((d,x)\bigr),\quad t\in[0,1].$$

By equation (\ref{omegaafin}) and Lemma \ref{samecenter}, we get that $$\omega\bigl(H(A,t)\bigr)=d\quad\textnormal{and}\quad\mathscr{C}\bigl(H(A,t)\bigr)=x$$ for every $t\in[0,1]$. Hence, $H$ is a well-defined contraction to the point $B\in\eta_D^{-1}\bigl((d,x)\bigr)$.

It remains to show that $\eta_D$ is proper. Let $K$ be a compact subset of $D\times\mathbb{R}^n$. Then the projections $\pi_D(K)$ and $\pi_{\mathbb{R}^n}(K)$ are compact subsets of $D$ and $\mathbb{R}^n$, respectively.

Let $\Gamma$ denote the compact set $\pi_D(K)\times\pi_{\mathbb{R}^n}(K)$. Then $\Gamma$ is a compact subset of $[0,\infty)\times\mathbb{R}^n$. By continuity of $\eta_{[0,\infty)}:cw_{[0,\infty)}(\mathbb{R}^n)\to[0,\infty)\times\mathbb{R}^n$ and Lemma \ref{closed}, we have that $\eta_{[0,\infty)}^{-1}(\Gamma)$ is closed in $cc(\mathbb{R}^n)$. We put $$\delta=\max\pi_D(K),\quad r=\max\bigl\{\|y\|\,\big|\,y\in\pi_{\mathbb{R}^n}(K)\bigr\}\quad\textnormal{and}\quad O=B(0,\delta+r).$$  Then $cc(O)$ is a compact subset of $cc(\mathbb{R}^n)$ (see \cite[p. 568]{Nadler2}) that contains $\eta_{[0,\infty)}^{-1}(\Gamma)$. Indeed, if $Y\in\eta_{[0,\infty)}^{-1}(\Gamma)$, then $\omega(Y)\leq\delta$ and $\|\mathscr{C}(Y)\|\leq r$. By \cite[Theorem~6]{Bohnenblust}, the radius $\mathscr{R}(Y)$ of $\mathscr{B}(Y)$ is less than $\omega(Y)$. Hence, $\mathscr{R}(Y)+\|\mathscr{C}(Y)\|\leq\delta+r$. Thus, we get that $$Y\subset\mathscr{B}(Y)\subset O.$$ It follows that $\eta_{[0,\infty)}^{-1}(\Gamma)$ is closed in $cc(O)$, and therefore, it is compact. Finally, by continuity of $\eta_D$, $\eta_D^{-1}(K)$ is closed in $\eta_D^{-1}(\Gamma)=\eta_{[0,\infty)}^{-1}(\Gamma)$, and thus, it is also compact. This completes the proof.
\end{proof}

Next, we state the main result of the section.

\begin{theorem}\label{main1}
Let $D\neq\{0\}$ be a convex subset of $[0,\infty)$. Then the hyperspace $cw_D(\mathbb{R}^n)$ is homeomorphic to $D\times\mathbb{R}^n\times Q$.
\end{theorem}

\begin{proof}
By Lemma \ref{Correction}, the hyperspace $cw_D(\mathbb{R}^n)$ is a $Q$-manifold. Clearly, $D\times\mathbb{R}^n$ is a locally compact ANR. By Proposition \ref{propCE}, the map $\eta_D:cw_D(\mathbb{R}^n)\to D\times\mathbb{R}^n$ defined by formula ($\ref{eta}$) is a cell-like map. Consequently, by Theorem \ref{Edwards}, $cw_D(\mathbb{R}^n)$ is homeomorphic to $D\times\mathbb{R}^n\times Q$.
\end{proof}

\begin{corollary}\label{maincor}
Let $D\neq\{0\}$ be a convex subset of $[0,\infty)$. Then
\begin{enumerate}
	\item [\textnormal{(1)}]$cw_D(\mathbb{R}^n)$ is homeomorphic to $\mathbb{R}^n\times Q$, if $D$ is compact,
	\item [\textnormal{(2)}] $cw_D(\mathbb{R}^n)$ is homeomorphic to $\mathbb{R}^{n+1}\times Q$, if $D$ is an open interval,
	\item [\textnormal{(3)}] $cw_D(\mathbb{R}^n)$ is homeomorphic to $Q_0:=Q\backslash\{*\}$, if $D$ is a half-open interval.
\end{enumerate}
In particular, the hyperspace $cw(\mathbb{R}^n)$ of all convex bodies of constant width is homeomorphic to $\mathbb{R}^{n+1}\times Q$.
\end{corollary}

\begin{proof}
By Theorem \ref{main1}, $cw_D(\mathbb{R}^n)$ is homeomorphic to $D\times\mathbb{R}^n\times Q$.

(1) If $D$ is compact, then $D$ is either a point or a closed interval. In either case, $D\times Q$ is homeomorphic to $Q$ and thus, $D\times\mathbb{R}^n\times Q$ is homeomorphic to $\mathbb{R}^n\times Q$.

(2) If $D$ is an open interval, then $D$ is homeomorphic to $\mathbb{R}$ and thus, $D\times\mathbb{R}^n\times Q$ is homeomorphic to $\mathbb{R}^{n+1}\times Q$.

(3) Let $D$ be a half open interval. Since $\mathbb{R}^n\times Q$ is a contractible $Q$-manifold, we have by \cite[Corollary 21.4]{Chapman} that $D\times\mathbb{R}^n\times Q$ is homeomorphic to $Q\times[0,1)$, which in turn is homeomorphic to $Q_0$ (see \cite[Theorem~12.2]{Chapman}).
\end{proof}

\begin{corollary}\label{coropen1}
If a subspace $U$ of $[0,\infty)$ can be represented as the topological sum $\bigoplus_{i\in I}D_i$ of a family $(D_i)_{i\in I}$ of pairwise disjoint non-empty convex subsets $D_i\neq\{0\}$ of $[0,\infty)$ (e.g., if $U$ is open in $[0,\infty)$), then the hyperspace $cw_U(\mathbb{R}^n)$ of all compact convex sets of constant width $u\in U$ is homeomorphic to $U\times\mathbb{R}^n\times Q$. %The topological possibilities are $$\bigsqcup_{i=1}^{k}\mathbb{R}_i^{n+1}\times Q\quad or\quad Q_0\sqcup\left(\bigsqcup_{i=1}^{k-1}\mathbb{R}_i^{n+1}\times Q\right),$$ with $1\leq k\leq\infty$ being the number of connected components of $U$. %Thus, for non-empty open subsets $U$ and $V$ of $[0,\infty)$ such that either $0\notin U\cup V$ or $0\in U\cap V$, the hyperspaces $cw_U(\mathbb{R}^n)$ and $cw_V(\mathbb{R}^n)$ are homeomorphic if and only if $U$ and $V$ have the same number of connected components.
\end{corollary}

\begin{proof}
Since the sets $D_i$, $i\in I$, are pairwise disjoint open subsets of $U$, the sets $cw_{D_i}(\mathbb{R}^n)$, $i\in I$, are also pairwise disjoint open subsets of $cw_U(\mathbb{R}^n)$. Moreover, since $cw_U(\mathbb{R}^n)$ is the disjoint union of the hyperspaces $cw_{D_i}(\mathbb{R}^n)$, $i\in I$, %$$ %=\bigsqcup_{i\in I}cw_{D_i}(\mathbb{R}^n)$$
we have the homeomorphism:
\begin{equation*}\label{cwuh}
cw_U(\mathbb{R}^n)\cong\bigoplus_{i\in I}cw_{D_i}(\mathbb{R}^n)
\end{equation*}
(see \cite[Corollary\,2.2.4]{Engelking}). Now, Theorem \ref{main1} implies that for every $i\in I$, the hyperspace $cw_{D_i}(\mathbb{R}^n)$ is homeomorphic to $D_i\times\mathbb{R}^n\times Q$ and, consequently, $$\bigoplus_{i\in I}cw_{D_i}(\mathbb{R}^n)\cong\bigoplus_{i\in I}(D_i\times\mathbb{R}^n\times Q)\cong\Bigl(\bigoplus_{i\in I}D_i\Bigr)\times\mathbb{R}^n\times Q=U\times\mathbb{R}^n\times Q.$$ This completes the proof.
\end{proof}

\section{The hyperspaces $crw_D(\mathbb{R}^n)$}

In this section we describe for every $n\geq1$, the topology of the hyperspaces $crw_D(\mathbb{R}^n)$. Here is the main result.

\begin{theorem}\label{main2}
Let $D\neq\{0\}$ be a non-empty convex subset of $[0,\infty)$. Then for every $n\ge 2$, the hyperspace $crw_D(\mathbb{R}^n)$ is homeomorphic to $cw_D(\mathbb{R}^n)$.
\end{theorem}

The proof of this theorem is given at the end of the section; it is preceded by a series of auxiliary lemmas. However, for  $n=1$ we have the following simple result.

\begin{proposition}\label{crw1}
Let $D\neq\{0\}$ be a non-empty convex subset of $[0,\infty)$. Then the hyperspace $crw_D(\mathbb{R})$ is homeomorphic to $D\times\mathbb{R}\times[0,1]$.
\end{proposition}

\begin{proof}
Let $\Delta_D=\bigl\{(d,a)\in D\times\mathbb{R}\,\big|\,a\leq2d\bigr\}$ and define a map $f:crw_D(\mathbb{R})\to\Delta_D\times\mathbb{R}$ by the rule: $$f\bigl([x,y],[v,z]\bigr)=\left((z-x,\,y-x),\,\frac{x+y}{2}\right),\quad\bigl([x,y],[v,z]\bigr)\in crw_D(\mathbb{R}).$$ By definition, $z-x=y-v\in D$ is just the width of the pair $\bigl([x,y],[v,z]\bigr)$. Then $\frac{x+y}{2}=\frac{v+z}{2}\in\mathbb{R}$ is the middle point of $[x,y]$ and of $[v,z]$. Also, $y-x\leq 2(z-x)$. Indeed, if not, then $$2z=z+z\geq z+v=y+x>2z$$ which is a contradiction. Thus, $f$ is a well-defined map. We claim that $f$ is a homeomorphism. Indeed, define a map $g:\Delta_D\times\mathbb{R}\to crw_D(\mathbb{R})$ by the rule: $$g\bigl((d,a),p\bigr)=\left(\left[p-\frac{a}{2},p+\frac{a}{2}\right],\left[p-\left(d-\frac{a}{2}\right),p+\left(d-\frac{a}{2}\right)\right]\right),\quad (d,a)\in\Delta_D,\quad p\in\mathbb{R}.$$ A simple calculation shows that $g$ is the inverse map of $f$. Thus, $f$ is a homeomorphism. Finally, it is easy to see that $\Delta_D$ is homeomorphic to $D\times[0,1]$. This completes the proof.
\end{proof}

For the rest of the section, we assume that $n\geq2$.

%Before we prove that $\mu_D$ is a cell-like map, we note that the proof of \cite[Theorem 3.1]{Bazylevych} also holds for the hyperspaces $crw_D(\mathbb{R}^n)$ whenever $D\neq\{0\}$, that is, the hyperspaces $crw_D(\mathbb{R}^n)$ are contractible $Q$-manifolds.

\begin{lemma}\label{closed2}
For every non-empty closed subset $K$ of $[0,\infty)$, the hyperspace $crw_K(\mathbb{R}^n)$ of all pairs of compact convex sets of constant width $k\in K$ is closed in $cc(\mathbb{R}^n)\times cc(\mathbb{R}^n)$.
\end{lemma}

\begin{proof}
Let $(Y_i,Z_i)_{i=1}^\infty$ be a sequence in $crw_K(\mathbb{R}^n)$ such that $(Y_i,Z_i)\leadsto(Y,Z)$, where $Y, Z\in cc(\mathbb{R}^n)$. Then, by \cite[Theorem 1.8.11]{Schneider}, we have that $h_{Y_i}\leadsto h_Y$ and $h_{Z_i}\leadsto h_Z$. Hence, we get that $w_{(Y_i,Z_i)}\leadsto w_{(Y,Z)}$ (see equation (\ref{widthpairmap})). Since for every $i\in\mathbb{N}$, $w_{(Y_i,Z_i)}$ is a constant map with value in $K$ and $K$ is closed, we infer that $w_{(Y,Z)}$ is also a constant map with value in $K$. Thus, $(Y,Z)\in crw_K(\mathbb{R}^n)$.
\end{proof}

\begin{lemma}\label{conv2}
If $D$ is a non-empty convex subset of $[0,\infty)$, then the hyperspace $crw_D(\mathbb{R}^n)$ is also convex with respect to the Minkowski operations.
\end{lemma}

\begin{proof}
Let $(Y,Z), (A,E)\in crw_D(\mathbb{R}^n)$, $t\in[0,1]$ and $u\in\mathbb{S}^{n-1}$. Then $$t(Y,Z)+(1-t)(A,E)=\bigl(tY+(1-t)A,tZ+(1-t)E\bigr)$$ and, by equality (\ref{suppconvex}), we have
\begin{align*}
  w_{(tY+(1-t)A,\,\,tZ+(1-t)E)}(u) &=h_{tY+(1-t)A}(u)+h_{tZ+(1-t)E}(-u) \\
   &=th_Y(u)+(1-t)h_A(u)+th_Z(-u)+(1-t)h_E(-u) \\
   &=t\bigl(h_Y(u)+h_Z(-u)\bigr)+(1-t)\bigl(h_A(u)+h_E(-u)\bigr) \\
   &=t\,w_{(Y,Z)}(u)+(1-t)\,w_{(A,E)}(u).
\end{align*}
Since $w_{(Y,Z)}$ and $w_{(A,E)}$ are constant maps with values in $D$, and since $D$ is convex, we get that $w_{(tY+(1-t)A,tZ+(1-t)E)}$ is also a constant map with value in $D$. Thus, the pair $t(Y,Z)+(1-t)(A,E)\in crw_D(\mathbb{R}^n)$.
\end{proof}

\begin{lemma}[{\rm\cite[Theorem 3.1]{Bazylevych2}}]\label{Correction2} Let $D\neq\{0\}$ be a non-empty convex subset of $[0,\infty)$. Then the hyperspace $crw_D(\mathbb{R}^n)$ is a contractible $Q$-manifold.
\end{lemma}

\begin{proof}
First we note that by formula (\ref{embedding2}) and Theorem \ref{main1}, the hyperspace $crw_D(\mathbb{R}^n)$ is infinite-dimensional. By Lemma \ref{conv2} it is also convex, and hence, contractible. It remains to show that $crw_D(\mathbb{R}^n)$ is a $Q$-manifold.

If $D$ is closed, then by Lemma \ref{closed2}, $crw_D(\mathbb{R}^n)$ is closed in $cc(\mathbb{R}^n)\times cc(\mathbb{R}^n)$, and therefore, it is locally compact. Then the map $\varphi\times\varphi$, defined by formula (\ref{embedding3}), embeds $crw_D(\mathbb{R}^n)$ as a locally compact closed convex subset in the Banach space $C(\mathbb{S}^{n-1})\times C(\mathbb{S}^{n-1})$. Thus, by \cite[Theorem 7.1]{BP}, $crw_D(\mathbb{R}^n)$ is homeomorphic to either $Q_0$ or $\mathbb{R}^m\times Q$ for some $0<m<\infty$. In either case, $crw_D(\mathbb{R}^n)$ is a $Q$-manifold.

Next, if $D$ is open, then $K_D:=[0,\infty)\backslash D$ is closed. By Lemma \ref{closed2}, $crw_{K_D}(\mathbb{R}^n)$ is closed in $cc(\mathbb{R}^n)\times cc(\mathbb{R}^n)$, and hence, also in $crw_{[0,\infty)}(\mathbb{R}^n)$. Equivalently, $crw_D(\mathbb{R}^n)$ is open in $crw_{[0,\infty)}(\mathbb{R}^n)$, which by the above paragraph is a $Q$-manifold. Thus, we infer that $crw_D(\mathbb{R}^n)$ is also a $Q$-manifold.

Finally, let $D$ be a half-open interval properly contained in $[0,\infty)$. Assume without loss of generality that $D=[a,b)$ with $b>a\geq0$. Then $D$ is open in $[a,b]$ and consequently, $crw_D(\mathbb{R}^n)$ is open in $crw_{[a,b]}(\mathbb{R}^n)$. Since  $crw_{[a,b]}(\mathbb{R}^n)$ is a $Q$-manifold, we infer that $crw_D(\mathbb{R}^n)$ is also a $Q$-manifold. This completes the proof.
\end{proof}

It was proved in \cite[Theorem 3]{Maehara} that if $(Y,Z)$ is a pair of constant relative width $d\geq0$, then the Minkowski sum $Y+Z$ is a compact convex set of constant width $2d$.  Moreover, we have the following proposition.

\begin{proposition}\label{propCE2}
For every non-empty convex subset $D$ of $[0,\infty)$, the map\linebreak$\Phi_D:crw_D(\mathbb R^{n})\to cw_D(\mathbb R^{n})$ defined by
$$\Phi_D\bigl((Y,Z)\bigr)=\frac{1}{2}(Y+Z),\qquad (Y,Z)\in crw_D(\mathbb{R}^n)$$
is a cell-like map.
\end{proposition}

\begin{proof}
It follows from formula (\ref{widths}), equality (\ref{suppconvex}) and \cite[Theorem 3]{Maehara} that the map $\Phi_D$ is well-defined and clearly, it is continuous. Let $Y\in cw_D(\mathbb{R}^n)$. Then the pair $(Y,Y)\in crw_D(\mathbb{R}^n)$ and $\Phi_D\bigl((Y,Y)\bigr)=Y$ (see \cite[Theorem~2.1.7]{Webster}). Hence, the map $\Phi_D$ is surjective. We claim that the inverse image $\Phi_D^{-1}(E)$ of every $E\in cw_D(\mathbb{R}^n)$ is convex and thus, contractible. Indeed, let $(A,B), (Y,Z)\in\Phi_D^{-1}(E)$ and $t\in[0,1]$. Then $$\frac{1}{2}(A+B)=\frac{1}{2}(Y+Z)=E.$$ By Lemma \ref{conv2},  $$t(A,B)+(1-t)(Y,Z)=\bigl(tA+(1-t)Y,tB+(1-t)Z\bigr)\in crw_D(\mathbb{R}^n).$$ Moreover,
\begin{align*}
  \Phi_D\Bigl(\bigl(tA+(1-t)Y,tB+(1-t)Z\bigr)\Bigr) &=\frac{1}{2}\bigl(tA+(1-t)Y+tB+(1-t)Z\bigr) \\
   &=\frac{1}{2}\bigl(t\,(A+B)+(1-t)(Y+Z)\bigr) \\
   &=t\,\frac{1}{2}(A+B)+(1-t)\frac{1}{2}(Y+Z) \\
   &=tE+(1-t)E=E.
\end{align*}
Therefore, $t(A,B)+(1-t)(Y,Z)\in\Phi_D^{-1}(E)$. Consequently, $\Phi_D^{-1}(E)$ is contractible. It remains to show that $\Phi_D$ is a proper map. Consider a compact subset $$\Gamma\subset cw_D(\mathbb R^{n})\subset cw_{[0,\infty)}(\mathbb R^{n}).$$ Observe that $\Phi_D$ is the restriction of $\Phi_{[0,\infty)}$ to $crw_D(\mathbb R^{n})$ and $\Phi_{[0,\infty)}\bigl((Y,Z)\bigr)\in cw_D(\mathbb R^{n})$ if and only if $(Y,Z)\in crw_{D}(\mathbb R^{n})$.  This implies that
$$\Phi_{D}^{-1}(\Gamma)=\Phi_{[0,\infty)}^{-1}(\Gamma)$$
is closed in $crw_{[0,\infty)}(\mathbb R^{n})$ and according to Lemma~\ref{closed2}, $\Phi_{D}^{-1}(\Gamma)$ is also closed in the product $cc(\mathbb{R}^n)\times cc(\mathbb{R}^n)$. Now, since $\Gamma$ is compact and $\omega:cw_{[0,\infty)}(\mathbb{R}^n)\to[0,\infty)$ is continuous (see (\ref{diam})), we can find a positive number $M\geq \max\,\omega(\Gamma):=d$
such that $$A\subset B(0, M)$$ for every $A\in\Gamma.$ Let $(Y,Z)\in \Phi_{D}^{-1}(\Gamma)$. Then $\Phi_D\bigl((Y,Z)\bigr)=\frac{1}{2}(Y+Z)\in \Gamma$. %and hence, $\omega\bigl(\frac{1}{2}(Y+Z)\bigr)\leq d$.
Consequently, the pair $(Y,Z)$ is of constant relative width $\omega\bigl(\frac{1}{2}(Y+Z)\bigr)\leq d$ and thus, $$\|y-z\|\leq d<2d$$ for every $y\in Y$ and $z\in Z$. On the other hand, by the choice of $M$, we infer that
$$\|y+z\|\leq 2M$$ for every $y\in Y$ and $z\in Z.$
Finally, using the parallelogram law we get that
$$\|y\|^2+\|z\|^{2}=\frac{1}{2}\bigl(\|y+z\|^2+ \|y-z\|^2\bigr)<\frac{4M^2+4d^2}{2}\leq 4M^{2}$$
for each $y\in Y$ and $z\in Z$. This directly implies that
$\|y\|<2M$ and $\|z\|<2M$, if $y\in Y$ and $z\in Z$. Therefore, we can infer  that $(Y,Z)\in cc\big(B(0,2M)\big)
\times cc\big(B(0,2M)\big)$ and hence
$\Phi_{D}^{-1}(\Gamma)\subset cc\big(B(0,2M)\big)
\times cc\big(B(0,2M)\big)$. Now, using the fact that  $\Phi_{D}^{-1}(\Gamma)$ is closed in $cc(\mathbb R^{n})\times cc(\mathbb R^{n})$ and $cc\big(B(0,2M)\big)
\times cc\big(B(0,2M)\big)$ is compact (see \cite[p. 568]{Nadler2}), we  conclude that $\Phi_{D}^{-1}(\Gamma)$ is compact. This completes the proof.
 \end{proof}

\medskip

\noindent
{\it Proof of Theorem~\ref{main2}.}\
By Lemmas \ref{Correction2} and \ref{Correction}, the hyperspaces $crw_D(\mathbb{R}^n)$ and $cw_D(\mathbb{R}^n)$ are $Q$-manifolds. %Hence, $cw_D(\mathbb{R}^n)$ is a locally compact ANR.
By Proposition \ref{propCE2}, $cw_D(\mathbb{R}^n)$ is a cell-like image of $crw_D(\mathbb{R}^n)$. Consequently, Theorem \ref{Edwards} implies that $crw_D(\mathbb{R}^n)$ is homeomorphic to $cw_D(\mathbb{R}^n)\times Q$, which in turn, by the Stability Theorem for $Q$-manifolds (\cite[Theorem~15.1]{Chapman}), is homeomorphic to $cw_D(\mathbb{R}^n)$.
\qed

%We end the paper with the following Remark. Recall that $2^{\mathbb{R}^n}$ denotes the hyperspace of all non-empty compact subsets of $\mathbb{R}^n$ endowed with the Hausdorff metric topology.

%\begin{remark}\label{fremark}
%Since every ellipsoid in $\mathbb{R}^n$ is the image of the euclidean unit ball $\mathbb{B}^n$ under a non-singular linear transformation $g:\mathbb{R}^n\to\mathbb{R}^n$, our Theorems \ref{main1} and \ref{main2} naturally remain valid for the hyperspaces of (pairs of) compact convex sets of constant (relative) width in arbitrary Euclidean spaces (see \cite{Eggleston1} and \cite{Sallee} for the corresponding definitions). Indeed, let $E$ be an ellipsoid in $\mathbb{R}^n$ and $g:\mathbb{R}^n\to\mathbb{R}^n$ a linear transformation such that $gE=\mathbb{B}^n$. Further, let $Ecw_D(\mathbb{R}^n)$ (resp., $Ecrw_D(\mathbb{R}^n)$) denote the subspace of $cc(\mathbb{R}^n)$ (resp., of $cc(\mathbb{R}^n)\times cc(\mathbb{R}^n)$) consisting of all (resp., pairs of) compact convex sets of constant (resp., relative) width $d\in D$ with respect to $E$. Then the hyperspace map $2^g$ is a self-homeomorphism of $2^{\mathbb{R}^n}$, which restricts to a self-homeomorphism of $cc(\mathbb{R}^n)$ (see \cite[Theorem~1.3]{Nadler1}). It follows easily that $2^g\bigl(Ecw_D(\mathbb{R}^n)\bigr)=cw_D(\mathbb{R}^n)$ and $(2^g\times2^g)\bigl(Ecrw_D(\mathbb{R}^n)\bigr)=crw_D(\mathbb{R}^n)$. Thus, by Theorems \ref{main1} and \ref{main2}, $Ecw_D(\mathbb{R}^n)$ and $Ecrw_D(\mathbb{R}^n)$ are both homeomorphic to $D\times\mathbb{R}^n\times Q$, as claimed.
%\end{remark}

\smallskip

\smallskip

\bibliographystyle{amsplain}

\end{document}